\documentclass[a4paper,11pt, reqno]{amsart}
\usepackage[latin1]{inputenc}

\usepackage{amssymb}
\usepackage{amsmath}
\usepackage{amsthm}
\usepackage{graphicx}
\usepackage{a4wide}
\usepackage{geometry}
\usepackage{amsmath}
\usepackage{amsfonts}
\usepackage{amssymb}
\usepackage{graphicx}%
\providecommand{\U}[1]{\protect\rule{.1in}{.1in}}

\renewcommand{\ll}{{\ell}}

\newcommand{\be}[1]{\begin{equation}\label{#1}}
\newcommand{\ee}{\end{equation}}

\usepackage{color}

\let\pa\partial
\let\r\rho

\newtheorem{thm}{Theorem}
\newtheorem{lem}{Lemma}
\newtheorem{prop}{Proposition}
\newtheorem{rem}{Remark}
\newtheorem{assumption}{Assumption}

\newcommand{\beq}{\begin{eqnarray}}
\newcommand{\eeq}{\end{eqnarray}}
\newcommand{\beqs}{\begin{eqnarray*}}
\newcommand{\eeqs}{\end{eqnarray*}}
\newcommand{\bequ}{\begin{equation}}
\newcommand{\eequ}{\end{equation}}

\def\r{\rho}

\def\T{{\cal T}}
\def\M{{\cal M}}
\newcommand{\cal}{\mathcal}

\date{
June 24, 2022
}

\let\pa\partial

\makeatother
\begin{document}
\title[]{Global well-posedness for the thermodynamically refined passively transported nonlinear moisture dynamics with phase changes}
\author{Sabine Hittmeir}
\address[S. Hittmeir]{Fakult\"at f\"ur Mathematik, Universit\"at Wien, Oskar-Morgenstern-Platz 1, 1090 Wien, Austria}
\email{sabine.hittmeir@univie.ac.at}
\author{Rupert Klein}
\address[R. Klein]{FB Mathematik \& Informatik, Freie Universit\"at Berlin, Arnimallee 6, 14195 Berlin, Germany}
\email{rupert.klein@math.fu-berlin.de}
\author{Jinkai Li}
\address{Jinkai Li, South China Research Center for Applied Mathematics and Interdisciplinary Studies, School of Mathematical Sciences, South China Normal University, Zhong Shan Avenue West 55, Guangzhou 510631, P. R. China}
\email{jklimath@m.scnu.edu.cn; jklimath@gmail.com}
\author{Edriss~S.~Titi}
\address[Edriss~S.~Titi]{
Department of Mathematics, Texas A\&M University, 3368 TAMU, College Station,
TX 77843-3368, USA. Department of Applied Mathematics and Theoretical Physics, University of Cambridge, Cambridge CB3 0WA, U.K.
ALSO, Department of Computer Science and Applied Mathematics, Weizmann Institute of Science, Rehovot 76100, Israel.}
\email{titi@math.tamu.edu; Edriss.Titi@maths.cam.ac.uk}

.

\maketitle
\begin{abstract}
In this work we study the global solvability of moisture dynamics with phase changes for warm clouds. We thereby in
comparison to previous studies \cite{HKLT} take into account the different gas constants for dry air and water vapor as
well as the different heat capacities for dry air, water vapor and liquid water, which leads to a much stronger
coupling of the moisture balances and the thermodynamic equation. This refined thermodynamic setting has been
demonstrated to be essential e.g. in the case of deep convective cloud columns in \cite{HK}. The more complicated
structure requires careful derivations of sufficient a priori estimates for proving global existence and uniqueness of
solutions.
\end{abstract}

\allowdisplaybreaks
\section{Introduction}
 Precipitation causes one of the major uncertainties in weather forecast and climate modelling and thus also the
 incorporation of moisture and phase changes into atmospheric flow models is still actively debated, see e.g.\,\cite{B}.
In the framework of systematically derived reduced mathematical models for atmospheric dynamics, it has been shown that not
only the inclusion of moist processes alone, but also the detailed structure of the moist-air submodels can decisively affect the
overall flow dynamics, see, e.g., \cite{SmithStechmann2017,BoualemsBook2019,StechmannHottovy2020}. Often the difference of the gas constants for water vapor and dry air is neglected and further the simple form of the dry
ideal gas law is assumed to hold. Even more typically also the dependence of the internal energy on the moisture
components is neglected, which results in a much simpler form of the thermodynamic equation. So far global well-posedness
of solutions to moisture models has only been proven based upon these assumptions, see also
\cite{BCZ,CZF,CZT,CZH,HKLT,HKLT2}. As demonstrated e.g.\,in the asymptotical analysis in \cite{HK} for deep convective
cloud columns, exactly these refined thermodynamics lead to a much  stronger coupling of the thermodynamic equation (see
below) to the moisture components and thereby even change the force balances to leading order. The aim here is to also
incorporate them into the analysis,  where the refined thermodynamical setting in comparison to \cite{HKLT} requires a different approach for proving a priori nonnegativity and uniform boundedness of the solution components, since the
antidissipative term in the equation for temperature does not vanish anymore when rewriting it in terms of the potential
temperature. We thus employ  an iterative method similar to the one used by Coti Zelati et al.\,\cite{CZH} to derive an upper bound on the temperature.

Coti Zelati et al.\,in \cite{BCZ,CZF,CZH,CZT} analysed a basic moisture model consisting of one moisture quantity coupled to temperature and containing only the process of condensation
during upward motion, see e.g.\,\cite{HW}. Since the source term there is modeled via a Heavy side
function as a switching term between saturated and undersaturated regions, the analysis requires elaborate techniques.  The approach based on differential inclusions and variational techniques has then further been applied to the moisture model coupled to the primitive equations in \cite{CZH}.

In preceding works \cite{HKLT, HKLT2} we studied a moisture model consisting of three moisture quantities for  water vapor, cloud water, and rain water, which contains besides the phase changes condensation and evaporation also the autoconversion of cloud water to rain water
after a certain threshold is reached, as well as  the collection of cloud water by the falling rain droplets. It
corresponds to a basic form of a bulk microphysics model in the spirit of Kessler \cite{Ke} and Grabowski and
Smolarkiewicz \cite{GS}. In \cite{HKLT} we assumed the velocity field to be given and studied the
moisture balances coupled to the thermodynamic equation through the latent heat. In \cite{HKLT2} this moisture model has been successfully coupled  to the primitive equations by taking over the ideas of Cao and Titi \cite{CT} for their breakthrough on
the global solvability of the latter system.

In this work we extend this moisture model for warm clouds consisting of three moisture balances and the thermodynamic equation by the refined thermodynamic setting as explained above, which leads in particular to a much stronger coupling of the model equations.

\color{black}

In the remainder of this section we introduce the moisture model. In Section \ref{sec.main} we then formulate the full problem with boundary and side conditions and state the main result on the global existence and uniqueness of bounded solutions. In Section \ref{sec.apx} we carry out the proof for the existence and uniqueness of strong solutions.

\subsection{Governing equations}

When modelling atmospheric flows in general,
the full compressible governing equations need to be considered. However, under the assumption of hydrostatic balance,
which in particular guarantees the pressure to decrease monotonically in height, the pressure $p$ can be used as the vertical coordinate, which have the  main advantage that the continuity equation takes the form of  the  incompressibility condition (see also \eqref{inc} below).
We therefore  work in the following with the governing equations in the pressure coordinates $(x,y,p)$ and as in \cite{HKLT} assume the velocity field to be given
$$
\overline{\mathbf{v}}=(\overline{\mathbf{v}}_h,\overline{\omega})=(\overline u, \overline v, \overline\omega),
$$
where we note that the vertical velocity $\omega=\tfrac{dp}{dt}$ in pressure coordinates takes the inverse sign in comparison to cartesian coordinates for upward and downward motion. Also the horizontal and the vertical derivatives  and accordingly the velocity components in pressure coordinates have different units. The total derivative  in pressure coordinates reads
\beq\label{tot.p}
\frac{d}{dt} = \pa_t + \overline{\mathbf{v}}_h\cdot \nabla_h + \overline{\omega}\pa_p \,, \qquad \textnormal{with} \quad \nabla_h =(\pa_x,\pa_y)\,.
\eeq
For the closure of the turbulent and molecular transport we use
\beq\label{tot.diff}
{\cal D}^*  =  \mu_*\Delta_h+ \nu_*\pa_p\left(\left(\frac{g p }{R_d \bar T}\right)^2\pa_p\right)\,,\qquad \textnormal{with} \quad \Delta_h=\pa_x^2+\pa_y^2\,,
\eeq
where $\bar T=\bar T(p)$ corresponds to some background distribution being uniformly bounded from above and below and $R_d$ is the individual gas constant for dry air. The operator ${\cal D}^*$  thereby provides a close approximation to the full Laplacian in cartesian coordinates, see also \cite{LTW,PTZ}.
The thermodynamic quanitities are related via
 the ideal gas law
\beq\label{ideal0}
p=p_d + p_v = R_d \r_d T + R_v \r_v T\,,
\eeq
where $R_v$ is the individual gas constant for water vapor and $p_d, p_v, \r_d, \r_v$ denote the partial pressures and
densities of dry air and water vapor, where we note that liquid water does not exert any pressure on the volume of air, see e.g. also \cite{B, C}.

 Before going more into details with the ideal gas law (see \eqref{ideal} below), we need to introduce the moisture quantities.
In the case of moisture being present typically  the water vapor mixing ratio, defined as the ratio of the density of $\r_v$ over the density
of dry air $\r_d$\,,
$$
q_{v}=\frac{\r_v}{\r_d}\,,
$$
is used for a measure of quantification.
If saturation effects occur, then water is also present in liquid form as cloud water and rain water represented by the additional moisture quantities
$$
q_{c}=\frac{\r_c}{\r_d}\,,\qquad q_{r}=\frac{\r_r}{\r_d}\,.
$$
We focus here on warm clouds, where water is present only in gaseous and liquid form, i.e. no ice and snow phases occur. The total water content
is therefore given by
$$
q_T=q_v+q_c+q_r\,.
$$
For these mixing ratios for water vapor, cloud water and rain water we have the following moisture balances
\beq
\label{eq.qv}\frac{d q_v}{dt}  &=& S_{ev} - S_{cd}+\cal{D}^{q_v} q_v\,,\\
\label{eq.qc}\frac{d q_c}{dt} &=& S_{cd} - S_{ac}- S_{cr}+\cal{D}^{q_c} q_c\,,\\
\label{eq.qr}\frac{d q_r}{dt}+V\pa_p\left(\frac{p}{R_d \bar{T}} q_r\right) &=& S_{ac} + S_{cr}- S_{ev}+\cal{D}^{q_r} q_r\,
\eeq
with $\frac{d}{dt}$ as in \eqref{tot.p} and $\cal{D}^{q_j} $ as in  \eqref{tot.diff}.
Here   $S_{ev}, S_{cd}, S_{ac}, S_{cr}$ are the rates of evaporation of rain water, the condensation of water vapor to cloud water and
the inverse evaporation process, the auto-conversion of cloud water into rainwater by accumulation of microscopic droplets,
and the collection of cloud water by falling rain. Moreover $V$ denotes the terminal velocity of falling rain and is assumed to be constant.

Having introduced the mixing ratios, we can now reformulate the ideal gas law \eqref{ideal0} as
\beq\label{ideal}
p=\r \widetilde{R} T\,,
\eeq
where $\rho=\rho_d+\rho_v+\rho_c+\rho_r$ is the total density and
$\widetilde{R}$ depends on the moisture content
$$
\widetilde{R}=R_d \frac{1+\frac{q_v}{E}}{1+q_T}\,,\qquad \textnormal{where} \qquad E=\frac{R_d}{R_v}\,,
$$
see e.g.\,also \cite{B,C,HK}.
The thermodynamic equation accounts for the diabatic source and sink terms, such as latent heating, radiation effects, etc., but we will in the following only focus on the effect of latent heat in association with phase
changes (see e.g.\,also \cite{KM,CZF,CZH,HKLT}).
The temperature equation in pressure coordinates then reads, see e.g.\,\cite{HK,C},
\beq
\label{eq.T}\frac{dT}{dt}-\widetilde{\kappa}\frac{T}{p} \overline{\omega} + \frac{c_l}{\widetilde{C}}q_r V \pa_p T=\widetilde{L}(S_{cd}-S_{ev})+ {\cal D}^TT\,,
\eeq
where
$$
\widetilde{\kappa}=\frac{\widetilde{R}}{\widetilde{C}}\,,\qquad \textnormal{and} \qquad \widetilde{C}=c_{pd} + c_{pv} q_v + c_l(q_c+q_r)
$$
with the heat capacities $c_{pd},c_{pv}$ at constant pressure for dry air and water vapor and the heat capacity for liquid water $c_l$, respectively.  For the latent heat term, we denote
\beq\label{L}
\widetilde{L}=\frac{L}{\widetilde{C}}\,,\qquad \textnormal{where}\qquad L(T)= L_0 - (c_{l}-c_{pv})(T-T_0) \qquad \textnormal{and} \qquad L_0=L(T_0),
\eeq
where $T_0$ is the reference temperature which is typically chosen as $T_0=273.15K$. We emphasize here once more that so far only models with $L, \widetilde{R}, \widetilde{\kappa}, \widetilde{C}$ constant and $c_l=0$ have been considered in mathematical analysis studies.
Thus this physically more refined sestting has several extensions in comparison to existing studies.
 \begin{rem}
To describe the state of the atmosphere a common thermodynamic quantity used instead of the temperature is the potential temperature
\begin{equation}\label{PT}
\theta = T \left(\frac{p_0}{p}\right)^\kappa\,, \qquad \textnormal{where} \qquad \kappa = \frac{R_d}{c_{pd}}\,.
\end{equation}
In case of the typical simplification  $\widetilde{\kappa}=\kappa$ and $c_l=0$,  the left hand side of \eqref{eq.T} simply reduces to $\frac{T}{\theta}\frac{d}{dt}\theta$. This property was essential in the preceding works \cite{HKLT} and also \cite{CZT} to derive a priori nonnegativity of the moisture quantities and temperature.
\end{rem}

\subsection{Explicit expressions for the source terms}

The saturation mixing ratio
\[q_{vs}=\frac{\r_{vs}}{\r_d}\,,\]
gives the threshold for saturation, i.e.\,$q_v<q_{vs}$ for undersaturation, $q_v = q_{vs}$ corresponds to saturation,
and  $q_v>q_{vs}$ accordingly holds in oversaturated regions. The saturation vapor mixing ratio satisfies
$$
q_{vs}(p, T)=\frac{E e_s(T)}{p-e_s(T)},
$$
with the saturation vapor pressure $e_s$ as a function of $T$ being defined by the Clausius-Clapeyron equation:
$$
\frac{d \ln e_s}{dT}=\frac{L(T)}{R_v T^2}\,.
$$
From this formula it is obvious that $e_s$ increases in $T$ (as long as $L(T)$ is positive).
Since the temperature $T$ is given in Kelvin $K$, only positive values are physical.
It should be noted that the Clausius-Clapeyron equation is only meaningful
for temperature ranges appearing in the troposphere, thus in particular
we shall pose in the following the natural assumption
$$
e_s(T)=0 \, \textit{Pa}\quad \textnormal{and} \quad q_{vs}(p,T)=0 \quad \textnormal{for} \  T \leq \underline{T}\,,
$$
for some $\underline{T}\geq 0\, \textit{K}$, which will also be helpful for proving nonnegativity of the moisture quantities and the temperature, see also \cite{HKLT}.

Recalling the fact that $c_l-c_{pv}>0$, we see from \eqref{L} that $L(T)$ decreases in $T$. In particular, there exists the critical temperature
\beq
\label{Tcrit}
T_{crit}=\frac{L_0}{(c_l-c_{pv})}-T_0
\eeq
at which the latent heat of evaporation vanishes, i.e.\,$L(T_{crit})=0$. At such high temperatures of about $700 K$, the gaseous and liquid state become indistinguishable. Such temperatures however clearly exceed by far the ones present in the relevant atmospheric layers. Therefore, we in the following pose the natural assumption that
\beq
\label{cond.qvs.up}
e_s(T)=0 \, \textit{Pa}\quad \textnormal{and} \quad q_{vs}(p,T)=0 \quad \textnormal{for} \  T \geq T_{crit}\,.
\eeq
For deriving the uniqueness of the solutions we need additionally the uniform Lipschitz continuity of $q_{vs}\geq 0$ in $T$, i.e.\,we assume
$$
|q_{vs}(p, T_1)-q_{vs}(p, T_2)|\leq C|T_1-T_2|,
$$
for a positive constant $C$ independent of $p$.

For the source terms of the mixing ratios,
we take over the setting of Klein and Majda \cite{KM} corresponding to a basic form of
the bulk microphysics closure in the spirit of Kessler \cite{Ke} and Grabowski and
Smolarkiewicz \cite{GS}, which has also been used in the preceding work \cite{HKLT}:
\begin{eqnarray*}
\label{Sev}S_{ev}&=&C_{ev}\widetilde{R}T (q_r^+)^\beta(q_{vs}-q_v)^+\,,\qquad \\
\label{Scr}S_{cr}&=&C_{cr} q_c q_r,\qquad \\
\label{Sac}S_{ac}&=&C_{ac} (q_c-q_{ac}^*)^+
\end{eqnarray*}
where $C_{ev},C_{cr},C_{ac}$ are dimensionless rate constants. Moreover, $(g)^+=\max\{0,g\}$
and $q_{ac}^*\geq0$ denotes the threshold for cloud water mixing ratio beyond which autoconversion of cloud water into precipitation become active.
The cutoff of the negative part in $q_r$ is only technical since clearly only nonnegative values for
$T$ and $q_j$ for $j\in\{v,c,r\}$ are meaningful.

The exponent $\beta$ in the evaporation term $S_{ev}$ in the literature typically appears to be chosen as
$\beta \approx 0.5$, see e.g.\,\cite{GS,KM} and the references therein.
Exponent $\beta \in (0,1)$ causes difficulties in the analysis for the uniqueness of the solutions. In the case that both $\widetilde C$ and $\widetilde R$ are constants, this problem, however, was overcome in \cite{HKLT} by introducing new unknowns, which allow for certain cancellation properties of the source terms and reveal advantageous monotonicity properties. Here, however, we need to generalise the setting to incorporate the more complicated structure of the thermodynamic equation and in particular the nonconstant $\widetilde C, \widetilde{L}$.

We shall use the closure of the condensation term in a similar fashion to \cite{KM}
$$\label{Scd}
S_{cd}=C_{cd}(q_v-q_{vs})q_c  +C_{cn}(q_v-q_{vs})^+\,,
$$
which is in the literature often defined implicitly via the equation of water vapor at saturation, see e.g.\,\cite{GS}.


\section{Formulation of the problem and main result}\label{sec.main}

We analyse the moisture model consisting of the moisture equations \eqref{eq.qv}--\eqref{eq.qr} coupled to the thermodynamic equation \eqref{eq.T}. As in \cite{HKLT}, we assume the velocity field $\bar{\mathbf{v}}=(\overline{\mathbf{v}}_h,\overline\omega)$ to be given and to satisfy
\begin{eqnarray}\label{reg.v}
  \overline{\mathbf{v}}_h \in (L^2_{\text{loc}}([0,\infty);H^1(\M)))^2\cap (L_{\text{loc}}^\infty([0,\infty);L^2(\M)))^2\cap (L^r_{\text{loc}}([0,\infty);L^q(\M)))^2,\\
  \overline\omega \in  L_{\text{loc}}^\infty([0,\infty);L^2(\M)) \cap L^r_{\text{loc}}([0,\infty);L^q(\M)),
\end{eqnarray}
for some $2\leq r\leq \infty$ and $3\leq q \leq \infty$ satisfying
$\frac{2}{r}+\frac{3}{q}<1$.
Moreover, we assume mass conservation, taking in pressure coordinates the form of the incompressibility condition
\beq\label{inc}
\nabla_h\cdot\overline{\mathbf{v}}_h+\partial_p\overline\omega =0 \qquad \textnormal{in} \ \ \M\,,
\eeq
and the no-penetration boundary condition
\beq\label{noflux.v}
\overline{\mathbf{v}}_h\cdot \mathbf{n}_h + \overline\omega\, n_p=0\,\qquad \textnormal{on} \ \ \pa\M\,.
\eeq
This is motivated from the solution of the viscous primitive equations (without moisture) satisfying these required regularity properties in \eqref{reg.v}, see \cite{CLT3,CLT,CLT1,CLT2,CT}.

Similar as in  \cite{CZF,HKLT}, we let ${\M}$ be a cylinder of the form
\beqs
{\M}=\{(x,y,p): (x,y) \in {\M}', p\in (p_1,p_0)\}\,,
\eeqs
where ${\M}'$ is a smooth bounded domain in $\mathbb{R}^2$ and $p_0>p_1>0$. The boundary is given by
\beqs
&&\Gamma_0=\{(x,y,p)\in \overline{{\M}}: p=p_0\}\,,  \\
&&\Gamma_1=\{(x,y,p)\in \overline{{\M}}: p=p_1\}\,,\\
&&\Gamma_{\ll} = \{(x,y,p)\in \overline{{\M}}: (x,y)\in \pa {\M}', p_0\geq p\geq p_1\}\,.
\eeqs
The boundary conditions read as
\beq
\label{bound.0}&\Gamma_0:\ 
&  \pa_pT=\alpha_{0T}(T_{b 0}(x,y,t)-T)\,,\quad \pa_pq_{j}=\alpha_{0j} (q_{b 0 j}(x,y,t)-q_j) \,, \quad j\in\{v,c,r\}\,,\qquad\\
&\Gamma_1:\  
&\pa_pT=0\,,\qquad \pa_pq_j=0\,,\quad j\in \{v,c,r\}\,, \\
\label{bound.ll}&\Gamma_{\ll}:\ 
& \pa_nT=\alpha_{{\ll}T}(T_{b \ll}(p,t)-T)\,,\quad \pa_nq_{j}=\alpha_{{\ll} j} (q_{b {\ll} j}(p,t)-q_j) \,, \quad j\in\{v,c,r\},
\eeq
where all given functions $\alpha_{0 j }, \alpha_{\ll j}, \alpha_{0 T }, \alpha_{\ll T}$ and
$T_{b 0 }, T_{b \ll},  q_{b 0 j }, q_{b \ll j }$ are assumed to be nonnegative, sufficiently smooth and uniformly bounded.

Throughout this paper, we use the abbreviation
\[\|f\|=\|f\|_{L^2(\M)}\,,\qquad \|f\|_{L^p}=\|f\|_{L^p(\M)}\,.\]
According to the weight in the vertical diffusion terms, we also  introduce the weighted norms
\[\|f\|_w=\Big\|\left(\frac{g p}{R_d\bar T}\right) f\Big\|\,,\qquad \|f\|^2_{H_w^1}=\|f\|^2+\|\nabla_h f\|^2+\|\pa_p f\|_w^2\,,\]
where we emphasize that, since the weight $\frac{gp}{R_d\bar T}$ is uniformly bounded from above and below by positive constants, the $H_w^1(\M)$-norm is equivalent to the $H^1(\M)$-norm.
Moreover, we shall often use for convenience the notation
\[\|(f_1,\dots,f_n)\|^2=\sum_{j=1}^n\|f_j\|^2.\]

For the initial data space of functions for the moisture components and the temperature,
we introduce
$$
\mathcal{X}=L^\infty(\M)\cap H^1(\M)
$$
and accordingly for strong solutions
$$
\mathcal{Y}_\T=\{f|f\in L^\infty(\M\times (0,\T))\cap L^2(0,\T;H^2(\M)),\partial_tf\in L^2(\mathcal M\times(0,\mathcal T))\}.
$$

The main result of this paper is the following theorem.

\begin{thm}
\label{thm}
Let $\beta=1$ and assume that the given velocity field $(\overline{\mathbf{v}}_h, \overline\omega)$ satisfies (\ref{reg.v})--(\ref{noflux.v}) and the initial data $(T_0,q_{v0},q_{c0},$ $q_{r0}) \in \mathcal{X}^4$ is nonnegative.
 Then, for any $\T>0$ there exists a unique global strong solution $( T,q_{v},q_{c},q_{r}) \in \mathcal{Y}_\T^4$ to system \eqref{eq.qv}--\eqref{eq.T}, subject to \eqref{bound.0}--\eqref{bound.ll}, on $\mathcal M\times(0,\mathcal T)$, and the solution components $(T,q_{v},q_{c},q_{r})$ remain nonnegative and uniformly bounded from above  with bounds growing with $\T$.
\end{thm}

The proof of this theorem will be presented in the next section.

\begin{rem}\label{REMARK}
In \cite{HKLT}, we treated the more complicated case of an evaporation source with a general exponent $\beta\in(0,1]$ of $q_r$ that causes in particular difficulties in the uniqueness. To overcome this problem, we introduced the new unknowns $Q=q_v+q_r$ and $H=T-\widetilde{L} (q_c+q_r)$ in \cite{HKLT},  where we recall that $\widetilde{L}$ was assumed to be constant there. Due to the challenge here of treating the additional terms arising from the refined thermodynamics, we stick to the case corresponding $\beta=1$ here and leave the more general case $\beta\in (0,1]$ for future work.
\end{rem}

Throughout this paper, we use $C$ to denote a general positive constant which may be different at different places.
For the aim of the future studies on the coupled system of the moisture dynamics investigated in the present paper
to the primitive equations with either isotropic or anisotropic dissipations,
see \cite{CLT3,CLT,CLT1,CLT2,CT}, the dependence
of the constant $C$ on the a priori bounds of the given velocity field will be explicitly pointed out
at the relevant places. However, the dependence of $C$ on the initial data or the parameters in the system will not
be paid attention to. We will also use $C_k, k\in\mathbb N$, to denote constants having relevant units.

\section{Proof of Theorem \ref{thm}}

\label{sec.apx}
As a start point, we consider the following modified system, which however
is equivalent to the original system \eqref{eq.qv}--\eqref{eq.T} for nonnegative solutions:
\begin{eqnarray}
&&\pa_tT+(\overline{\mathbf{v}}_h\cdot\nabla_h)T+\overline\omega\partial_pT-\widetilde{\kappa}^+\frac{T}{p} \overline\omega+\frac{c_lq_r}{\widetilde{C}^+}V\pa_pT=\widetilde{L}^+(S_{cd}^+-S_{ev}^+)+ {\cal D}^TT,\label{AT}\\
&&\partial_tq_v+\overline{\mathbf{v}}_h\cdot\nabla_hq_v+\overline\omega\partial_pq_v  =
S_{ev}^+
- S_{cd}^++\cal{D}^{q_v} q_v\,,\label{Aqv}\\
&&\partial_tq_c+\overline{\mathbf{v}}_h\cdot\nabla_hq_c+\overline\omega\partial_pq_c= S_{cd}^+ -S_{ac}^+- S_{cr}^++\cal{D}^{q_c} q_c\,,\label{Aqc}\\
&&\partial_tq_r+\overline{\mathbf{v}}_h\cdot\nabla_hq_r+\overline\omega\partial_pq_r+V\pa_p
\left(\frac{p}{R_d\bar{T}} q_r\right) = S_{ac}^+ + S_{cr}^+-
S_{ev}^++\cal{D}^{q_r} q_r\,,\label{Aqr}
\end{eqnarray}
where
\begin{eqnarray*}
  &{\widetilde\kappa}^+=\frac{{\widetilde R}^+}{{\widetilde C}^+}, \quad{\widetilde L}^+=\frac{L(T)}{\widetilde C^+},\quad \widetilde{R}^+=\frac{ R_d+R_vq_v^+}{1+q_v^++q_c^++q_r^+},\quad\widetilde C^+=c_{pd}+c_{pv}q_c^++c_l(q_c^++q_r^+),\\
&S_{ev}^+=C_{ev}  \widetilde{R}^+ T^+q_r^+(q_{vs}(p,T)-q_v)^+,\quad S_{cr}^+=C_{cr}q_c^+q_r^+,\quad S_{ac}^+=S_{ac}=C_{ac}(q_c-q_{ac}^*)^+,\\
&S_{cd}^+=C_{cd}(q_v^+-q_{vs}(p,T))q_c^++C_{cn}(q_v-q_{vs}(p,T))^+,
\end{eqnarray*}
with $L(T)$ given by (\ref{L}).

Since all the nonlinear terms $S_{ev}^+, S_{cr}^+,
S_{ac}^+, S_{cd}^+$ and all the coefficients $\widetilde \kappa^+, \frac1{\widetilde C^+}, \widetilde L^+$ are Lipschitz with respect to $q_v, q_c, q_r,$ and $T$, the local existence of strong solutions to the initial boundary value problem of system (\ref{AT})--(\ref{Aqr}) follows by the standard fixed point arguments.
In fact, by following the proof in \cite{HKLT}, we can prove the following proposition on the local existence and uniqueness.

\begin{prop}
  \label{Aloc}
Assume that the given velocity field $(\overline{\mathbf{v}}_h, \overline\omega)$ satisfies (\ref{reg.v})--(\ref{noflux.v}) and the initial data $(T_0,q_{v0},q_{c0},$ $q_{r0}) \in \mathcal{X}^4$ is nonnegative. Then,
there exists a positive time $\mathcal T_0$ depending only on
the upper bound of $\|(T_0, q_{v0}, q_{c0}, q_{r0})\|_{H^1(\mathcal M)}$,
such that system (\ref{AT})--(\ref{Aqr}), subject to
(\ref{bound.0})--(\ref{bound.ll}), on $\mathcal M\times(0,\mathcal T_0)$, has a unique strong solution
$( T,q_{v},q_{c},q_{r}) \in \mathcal{Y}_{\T_0}^4$.
\end{prop}

By applying Proposition \ref{Aloc} inductively, one can extend uniquely
the solution $( T, q_v, q_c, q_r)$ obtained there to the maximal time interval $(0,\mathcal T_\text{max})$, where $T_{\text{max}}$ is characterized as
\begin{equation}\label{T_max}
  \limsup_{\mathcal T\rightarrow\mathcal T_{\text{max}}^-}\|( T, q_v, q_c, q_r)\|_{H^1(\mathcal M)}=\infty,\quad \mbox{if } \mathcal T_{\text{max}}<\infty.
\end{equation}
Observe that if $T_\text{max}=\infty$, then Proposition \ref{Aloc} implies Theorem \ref{thm}. Therefore, our aim is to show that $T_\text{max}=\infty$. To this end, we assume by contradiction that $T_\text{max}<\infty$. Due to this fact, the following assumption will be made in the subsequent propositions throughout this section.
\begin{assumption}\label{ass} Let all the assumptions in Proposition \ref{Aloc} hold, and let the solution $(T, q_v, q_c, q_r)$ obtained in Proposition \ref{Aloc} be extended uniquely to the maximal interval of existence $(0,\mathcal T_{\text{max}})$, where $\mathcal T_{\text{max}}<\infty$.
\end{assumption}

The main part of this section is to carry out a series of a priori estimates on $( T, q_v, q_c, q_r)$.

First, the following proposition about the nonnegative and uniform boundedness of the moisture components $q_v, q_c,$ and $q_r$ can be proved by slightly modifying the corresponding proof of Proposition 3.2 in \cite{HKLT}.

\begin{prop}\label{bddq}
Let Assumption \ref{ass} hold, then the solution $(T,q_{v},q_{c},q_{r})$ satisfies
$$
0\leq q_{v}\leq  q_v^*\,,\quad 0\leq q_{c}\leq q_c^*\,,\quad 0\leq q_{r}\leq  q_r^*,
$$
for any $\mathcal T \in (0,\mathcal T_{max})$, where
\begin{eqnarray*}
&&q_v^*=\max\big\{\|q_{v 0}\|_{L^\infty{(\M)}},\|q_{b0v}\|_{L^\infty((0,\T)\times\M')},\|q_{b\ll v}\|_{L^\infty((0,\T)\times\Gamma_\ll)}, q_{vs}^*\big\},\\
&&q_c^*\, =q_c^*\big(\T, \|q_{c 0}\|_{L^\infty{(\M)}},\|q_{b0c}\|_{L^\infty((0,T)\times\M')},\|q_{b\ll c}\|_{L^\infty((0,T)\times\Gamma_\ll)},q_v^*,q_{vs}^*\big)\,,\\
&&q_r^*\, =q_r^*\big(\T, \|q_{r 0}\|_{L^\infty{(\M)}},\|q_{b0r}\|_{L^\infty((0,T)\times\M')},\|q_{b\ll r}\|_{L^\infty((0,T)\times\Gamma_\ll)},q_c^*\big),
\end{eqnarray*}
with $q_{vs}^*=\max q_{vs}$ and moreover $q_c^*$ and $q_r^*$ are continuous in $\mathcal T\in(0,\infty)$.
\end{prop}

Due to the nonnegativity of $q_v, q_c, q_r$, it is clear that $\widetilde R^+=\widetilde R, \widetilde C^+=\widetilde C, \widetilde\kappa^+=\widetilde\kappa, \widetilde L^+=\widetilde L, S_{cr}^+=S_{cr}, S_{cd}^+=S_{cd}$, $S_{ac}^+=S_{ac}$,
and
\beq\label{bound.kappa}
0< \widetilde{\kappa} = \frac{\widetilde{R}}{\widetilde{C}}  \leq \kappa_1,\quad
0\leq \frac{c_l q_r}{\widetilde{C}}\leq 1,\quad0<\frac{1}{\widetilde C}\leq\frac1{c_{pd}},
\eeq
for some positive constant $\kappa_1$. Besides, by the uniform boundedness of $q_v, q_c, q_r$, one has
\begin{equation}\label{bound.source}
  |S_{cd}|+|S_{cr}|+|S_{ac}|\leq C,
\end{equation}
for some positive constant $C$ depending on $q_v^*, q_c^*, q_r^*$.

\begin{prop}
\label{H1q}
Let Assumption \ref{ass} hold, then for any $\mathcal T\in(0,T_\text{max})$,
\begin{equation*}
\sup_{0\leq t\leq \T} \|(q_v,q_c,q_r)\|^2 + \int_0^{\mathcal T}
\|\nabla (q_v,q_c,q_r)\|^2 dt\leq K_0(\mathcal T),
\end{equation*}
for a continuous bounded function $K_0(\T)$ determined by $q_v^*, q_c^*,$ and $q_r^*$.
\end{prop}

\begin{proof}
Testing (\ref{Aqv}), (\ref{Aqc}), and
(\ref{Aqr}), respectively, with $q_v, q_c,$ and $q_r$, summing the resultants up, using the uniform boundedness of
$q_v, q_c, q_r, S_{ac}, S_{cr}, S_{cd}$ (due to Proposition \ref{bddq} and (\ref{bound.source})), and noticing that $S_{ev}^+q_r\geq0$, one deduces
\begin{eqnarray}
 &&\frac12 \frac{d}{dt}\|(q_v, q_c, q_r)\|^2-\sum_{j\in\{v,c,r\}}\int_\M q_j\cal D^{q_j}q_jd\M\nonumber\\
 &=&-\sum_{j\in\{v,c,r\}}\int_\M(\overline{\mathbf{v}}_h\cdot\nabla_hq_j+\overline\omega\partial_pq_j)q_jd-V\int_\M q_r\partial_p\left(\frac{p q_r}{R_d\bar T}\right)d\M\nonumber\\
 && +\int_\M[
 (S_{ev}^+-S_{cd})q_v + (S_{cd}-S_{ac}-S_{cr})q_c + (S_{ac}+S_{cr}-S_{ev}^+)q_r]d\M\nonumber\\
 &\leq&-\sum_{j\in\{v,c,r\}}\int_\M(\overline{\mathbf{v}}_h\cdot\nabla_hq_j+\overline\omega\partial_pq_j)q_jd-V\int_\M q_r\partial_p\left(\frac{p q_r}{R_d\bar T}\right)d\M \nonumber\\
 &&+\int_\M S_{ev}^+ q_vd\M+C.\label{AD1}
\end{eqnarray}

The integrals in (\ref{AD1}) are estimated as follows.
For the diffusion terms, integration
by parts and using the boundary conditions (\ref{bound.0})--(\ref{bound.ll}),
one deduces
\begin{eqnarray*}
&&-\int_\M q_j{\cal D}^{q_j}q_jd\M\nonumber\\
&=&-\int_\M\left[\mu_{q_j}\Delta_hq_j+\nu_{q_j}\partial_p
\left(\left(\frac{gp}{R\bar T}\right)^2\partial_pq_j\right)\right]q_jd\M\nonumber\\
&=&-\mu_{q_j}\int_{\Gamma_\ll}(\partial_nq_j)q_jd\Gamma_\ll-\nu_{q_j}\left.\int_{\M'}
\left(\frac{gp}{R\bar T}\right)^2(\partial_pq_j)q_jd\M'\right|_{p_1}^{p_0}\nonumber\\
&&+\int_\M\left[\mu_{q_j}\nabla_hq_j\cdot\nabla_hq_j+\nu_{q_j}\left(\frac{gp}{R\bar T}\right)^2\pa_pq_j\partial_pq_j\right]d\M\nonumber\\
&=&\mu_{q_j}\|\nabla_hq_j\|^2+\nu_{q_j}\|\partial_pq_j\|_w^2-\mu_{q_j} \int_{\Gamma_\ll}\alpha_{\ll j}(q_{\ll j}-q_j)q_jd\Gamma_\ll\nonumber\\
&&-\nu_{q_j}\int_{\M'}\left(\frac{gp_0}{R\bar T}\right)^2\alpha_{b0j}(q_{b0j}-q_j)q_jd\M'\nonumber\\
&=&\mu_{q_j}\|\nabla_hq_j\|^2+\nu_{q_j}\|\partial_pq_j\|_w^2+\mu_{q_j} \int_{\Gamma_\ll}\alpha_{\ll j}\left[
\left(q_j-\frac{q_{\ll j}}{2}\right)^2-\frac{q_{\ll j}^2}{4}\right]d\Gamma_\ll\nonumber\\
&&+\nu_{q_j}\int_{\M'}\left(\frac{gp_0}{R\bar T}\right)^2\alpha_{b0j}\left[\left(q_j-\frac{q_{b0j}}2\right)^2-\frac{q_{b0j}^2}4\right]d\M',\label{AD2}
\end{eqnarray*}
for $j\in\{v,c,r\}$. This implies for $j\in\{v,c,r\}$ that
\begin{equation}
  -\int_\M q_j{\cal D}^{q_j}q_jd\M
\geq\mu_{q_j}\|\nabla_hq_j^-\|^2+\nu_{q_j}\|\partial_pq_j^-\|_w^2-C, \label{AD3}
\end{equation}
for a constant $C$ depending only on the given inhomogeneous boundary functions $\alpha_{0 j }, \alpha_{\ll j}$ and
$q_{b 0 j }, q_{b \ll j }$.
The integral containing the advection term vanishes due to \eqref{inc} and \eqref{noflux.v}, since
\beq
 &&\int_\M (\overline{\mathbf{v}}_h\cdot\nabla_h q_j + \overline\omega \pa_p q_j) q_jd\M= -\frac12\int_\M (\overline{\mathbf{v}}_h\cdot\nabla_h + \overline\omega \pa_p) (q_j)^2 d\M\nonumber\\
&=& -\frac12\int_{ \pa\M}(\overline{\mathbf{v}}_h\cdot\mathbf{n}_h + \overline\omega n_p)(q_j)^2d(\pa\M)+\frac12\int_{ \M}(q_j)^2(\nabla_h\cdot \overline{\mathbf{v}}_h+ \pa_p\overline\omega )d\M =0\,.\qquad\label{AD4}
\eeq
The Young inequality leads to
\begin{eqnarray}
  -V\int_\M q_r\partial_p\left(\frac{p q_r}{R_d\bar T}\right)d\M
  \leq\frac{\nu_{q_r}}{8}\|\partial_pq_r\|_w^2+C\|q_r\|^2. \label{AD5}
\end{eqnarray}
To estimate the term $\int_\M S_{ev}^+ q_vd\M$, we decompose the domain $\M$ as $\M=\M_+(t)\cup\M_-(t)$, where
$\M_+(t)=\{(x,y,p)\in\M|T(x,y,p,t)\geq T_{\text{crit}}\}$ and $\M_-(t)=\M\setminus\M_+(t).$
Due to (\ref{cond.qvs.up}), one can check that $S_{ev}^+=0$ on $\M_+(t)$, while on $\M_-(t)$, due to the
nonnegativity and uniform boundedness of $q_v, q_c, q_r$ guaranteed by Proposition \ref{bddq}, one has $S_{ev}^+\leq C$.
Therefore, we always have
$$
\label{bound.Sev}
0\leq S_{ev}^+\leq C\quad\mbox{on }\M
$$
and, as result, it holds that $\int_\M S_{ev}^+ q_vd\M\leq C$. Thanks to this and combining (\ref{AD3})--(\ref{AD5}),
the conclusion follows from (\ref{AD1}) by the Gr\"onwall inequality.
\end{proof}

We would like to point out that the a priori estimates obtained
in Proposition \ref{bddq} and Proposition \ref{H1q} do not depend on the a priori bounds of the given velocity $(\overline{\mathbf v}_h, \overline\omega)$.

The following lemma will be used later.

\begin{lem}\label{aux.lem.uni}
	Let $f \in L^\infty( 0,\T;L^2(\M))$ and $g,h \in L^2(0,\T;H^1(\M))\cap L^\infty(0,\T;L^2(\M))$. Then for some arbitrary $\delta_g,\delta_h >0$ at a.e. $t\in (0,\T)$ the following estimate holds
$$
	\left|\int_\M f \,g \,h\, d\M\right|\leq \delta_g \|\nabla g\|^2_{L^2(\M)} + \delta_h \|\nabla h\|^2_{L^2(\M)} +C (\|g\|^2_{L^2(\M)} +\|h\|^2_{L^2(\M)} )\,,
$$
	where $C=C(\delta_1,\delta_2,\|f\|_{L^\infty(0,\T;L^2(\M))})$.
\end{lem}
\begin{proof} We first bound the integral by
$$
	\left|\int_\M f \,g \,h\, d\M\right|\leq \|f\|_{L^2(\M)}\|g\,h \|_{L^2(\M)}\leq C (\|g^2\|_{L^2(\M)}+\|h^2\|_{L^2(\M)})
$$
	where $C$ depends on $\|f\|_{L^\infty(0,\T;L^2(\M))}$. We next employ the Gagliardo Nirenberg inequality, see e.g.\cite[Theorem 10.1]{Fri} and \cite[Theorem 1.1.4]{Zhe}, to estimate
\begin{eqnarray*}
	\|g^2\|_{L^2(\M)}&=&\|g\|^2_{L^4(\M)}\leq C\|g\|^{2\vartheta}_{L^2(\M)}\|\nabla g\|^{2(1-\vartheta)}_{L^2(\M)} + \|g\|^2_{L^2(\M)} \\
	&\leq& \delta_g \|\nabla g\|^2_{L^2(\M)} +  C(\delta_g) \|g\|^2_{L^2(\M)}
\end{eqnarray*}
	which holds for $\vartheta = \frac14$. In the last estimate we used Young's inequality. The same estimate holds for $h$, which concludes the proof.
\end{proof}

Nonnegativity of the temperature is proved in the following proposition.

\begin{prop}\label{nonneg.T}
Let Assumption \ref{ass} hold, then the temperature $T$ is nonnegative.
\end{prop}

\begin{proof}
As already mentioned above, for the refined thermodynamic modelling used in this work, the anti-dissipative term $\widetilde{\kappa}^+\frac{T}{p} \overline\omega$ in the thermodynamic equation (\ref{AT}) does not vanish anymore, when switching to the potential temperature $\theta$.
We therefore perform the estimate directly from the equation for $T$.

Testing the equation for $T$ with $-T^-$ yields
\beq
&&\frac12\frac{d}{dt}\int_\M (T^-)^2d\M +\int_\M T^-{\cal D}^{T}Td\M
-\int_\M(\overline{\mathbf{v}}_h\cdot\nabla T+\overline\omega\partial_pT)T^-d\M\nonumber\\
&= &\int_\M \frac{\widetilde{\kappa}}{p}(T^-)^2\overline{\omega} d\M -\int_\M \frac{c_l q_r}{\widetilde{C}} V T^- \pa_p  T^- d\M
-\int_\M \widetilde{L}T^-(S_{cd}-S^+_{ev}) d\M .\label{T.nonneg.est}
\eeq

The various terms in the above equality are estimated as follows.
For the diffusion term, integration
by parts and using the boundary conditions (\ref{bound.0})--(\ref{bound.ll}),
one deduces
\begin{eqnarray*}
\int_\M T^-{\cal D}^{T}Td\M&=&\int_\M\left[\mu_{T}\Delta_hT+\nu_{T}\partial_p
\left(\left(\frac{gp}{R\bar T}\right)^2\partial_pT\right)\right]T^-d\M\nonumber\\
&=&\mu_{T}\int_{\Gamma_\ll}(\partial_nT)T^-d\Gamma_\ll+\nu_{T}\left.\int_{\M'}
\left(\frac{gp}{R\bar T}\right)^2(\partial_pT)T^-d\M'\right|_{p_1}^{p_0}\nonumber\\
&&-\int_\M\left[\mu_{T}\nabla_hT\cdot\nabla_hT^-+\nu_{T}\left(\frac{gp}{R\bar T}\right)^2\pa_pT\partial_pT^-\right]d\M\nonumber\\
&=&\mu_{T}\|\nabla_hT^-\|^2+\nu_{T}\|\partial_pT^-\|_w^2+\mu_{T} \int_{\Gamma_\ll}\alpha_{\ll T}(T_{{b\l}}-T)T^-d\Gamma_\ll\nonumber\\
&&+\nu_{T}\int_{\M'}\left(\frac{gp_0}{R\bar T}\right)^2\alpha_{0T}(T_{b0}-T)T^-d\M'.
\end{eqnarray*}
Since the functions $\alpha_{\ll T}, T_{{b\l}}, \alpha_{0T}, T_{b0}$ are nonnegative and
$T T^- =- ( T^-)^2$, the last two boundary integrals
are nonnegative, and we obtain
\beq\label{est.diff.q}
  \int_\M T^-{\cal D}^{T}Td\M
\geq\mu_{T}\|\nabla_hT^-\|^2+\nu_{T}\|\partial_pT^-\|_w^2.
\eeq
Same as (\ref{AD4}), the integral containing the advection terms vanishes due to \eqref{inc} and \eqref{noflux.v}.
To bound the first term on the right-hand side containing the vertical velocity component,
we apply Lemma \ref{aux.lem.uni} and use (\ref{bound.kappa}) as follows
$$
\left| \int_\M \frac{\widetilde{\kappa}}{p}(T^-)^2\overline{\omega} d\M\right|\leq C \int_\M |\overline{\omega}|(T^-)^2 d\M \leq \frac{\mu_T}{4}\|\nabla_h T^-\|^2 + \frac{\nu_T}{4}\|\pa_pT^-\|^2_w +  C\|T^-\|^2,
$$
where $C$ depends on $\|\overline{\omega}\|_{L^\infty(0,\T;L^2(\M))}$.
Using (\ref{bound.kappa}) again and by the Young inequality, we can estimate the second term as
$$
\left| \int_\M \frac{c_l q_r}{\widetilde{C}} V T^- \pa_p  T^- d\M\right| \leq \frac{\nu_T}{4}\|\pa_p T^-\|_w^2 + C\|T^-\|^2\,.
$$
It remains to bound the integral with the latent heating terms.
\begin{eqnarray*}
&&-\int_\M \widetilde{L}T^-S_{cd} d\M \nonumber\\
&=& - \int \frac{L_0 +(c_l-c_{pv} )T_0}{\widetilde C}[C_{cd} (q_v-q_{vs}) q_c + C_{cn} (q_v-q_{vs})^+]T^- d\M\nonumber\\
&&
- \int_\M \frac{c_l-c_{pv}}{\widetilde{C}} (T^-)^2  [C_{cd} (q_v-q_{vs}) q_c + C_{cn} (q_v-q_{vs})^+] d\M
\leq   0\,,
\end{eqnarray*}
where we used the fact $c_{pv}<c_l$, $q_{vs}=0$ for $T\leq 0$, and the nonnegativity of all moisture quantities.
The integral term with the evaporation term in (\ref{T.nonneg.est}) vanishes since $S_{ev}^+ T^-=C_{ev}\widetilde RT^+q_r(q_{rs}(p,T)-q_v)^+T^-=0$.

Combining all bounds above, we thus obtain from (\ref{T.nonneg.est}) that
$$
\frac{1}{2}\frac{d}{dt} \|T^-\|^2  + \frac{\mu_T}{2}\|\nabla_h T^-\|^2 + \frac{\nu_T}{2}\|\pa_p T^-\|_w^2 \leq C \|T^-\|^2,
$$
and we can conclude by the Gr\"onwall inequality the nonnegativity of $T$ since $T_0^-=0$.
\end{proof}

From now on, the a priori estimates to be carried out depend on the a priori bounds of the given velocity $(\overline{\mathbf{v}}_h,\overline\omega)$. The relevant bounds of the velocity on which the solutions depend will be explicitly pointed out in the statements of the propositions.

\begin{prop}\label{L2T}
Let Assumption \ref{ass} hold, then for any $\T\in(0,\T_\text{max})$
$$
\sup_{0\leq t\leq  \T} \| T \|^2(t)+ \int_0^\T\| \nabla T \|^2(t) dt \leq K_1(\mathcal T),
$$
for a continuous bounded function $K_1(\T)$ determined by the quantities $\|\overline\omega\|_{L^\infty(0,\T;L^2(\M))}$, $q_v^*, q_c^*, q_r^*$.
\end{prop}

\begin{proof}
By testing the temperature equation with $T$ and by the same calculations as (\ref{AD3}) and (\ref{AD4})
to the diffusion terms and the convection terms, we have the following estimate
\begin{eqnarray*}
&&\frac{1}{2}\frac{d}{dt}\|T\|^2 + \mu_T\|\nabla_h T\|^2 +\nu_T\|\pa_pT\|_w^2\nonumber \\
&\leq& \int_\M \frac{\widetilde{\kappa}}{p} \overline{\omega} T^2 d\M
- \int_\M \frac{c_l q_r}{\widetilde{C}}V \pa_p T \, T d\M + \int_\M \widetilde{L}(S_{cd}-S_{ev})  Td\M+C.
\end{eqnarray*}
To bound the first term on the right-hand side, we use Lemma \ref{aux.lem.uni}
and (\ref{bound.kappa}) to get
$$
\int_\M \frac{\widetilde{\kappa}}{p} \overline{\omega} T^2 d\M \leq C \int_\M |\overline{\omega}| T^2 d\M \leq
 \frac{\mu_T}{4}\|\nabla_h T\|^2 +\frac{\nu_T}{4}\|\pa_pT\|_w^2
+ C\|T\|^2,
$$
where $C$ depends on $\|\overline{\omega}\|_{L^\infty(0,\T;L^2(\M))}$. By Young's inequality and using (\ref{bound.kappa}) again,
we obtain
$$
\left|\int_\M \frac{c_l q_r}{\widetilde{C}}V \pa_p T \, T d\M \right| \leq \frac{\nu_T}{4}\|\pa_p T\|^2_w + C\|T\|^2.
$$
Recalling the expression of $L(T)$ given by (\ref{L}), we obtain due to (\ref{bound.kappa}) and (\ref{bound.source}) that
$$
\left|\int_\M \widetilde{L}S_{cd} T d\M\right|=\left|\int_\M \frac{L(T)}{\widetilde C} S_{cd}Td\M\right|\leq C (C_1+\|T\|^2).
$$
Since (\ref{cond.qvs.up}) implies
$S_{ev}=C_{ev}\widetilde RTq_r(q_{vs}(p,T)-q_v)^+=C_{ev}\widetilde RTq_r(-q_v)^+=0$, for $T>T_{\text{crit}}$, while (\ref{L})
and (\ref{Tcrit}) lead to $L(T)\geq0$, for $0\leq T\leq T_{\text{crit}}$, we therefore have
\begin{equation}
  \widetilde L S_{ev}=\frac{L(T)}{\widetilde C}S_{ev}\geq0 \label{EQ-1}
\end{equation}
and thus
$$
-\int_\M \widetilde{L}S_{ev} T d\M \leq 0\,.
$$
Combining all above estimates, we obtain
$$\label{dL2T}
\frac{1}{2}\frac{d}{dt}\|T\|^2 + \mu_T\|\nabla_h T\|^2 +\nu_T\|\pa_p T\|_w^2 \leq C(C_2+\|T\|^2),
$$
leading to the conclusion by the Gr\"onwall inequality.
\end{proof}

Uniform boundedness of $T$ is stated and proved in the following proposition.

\begin{prop}\label{LinftyT}
Let Assumption \ref{ass} hold, then for any $\T\in(0,\T_\text{max})$
$$
\sup_{0\leq t\leq\T}\|T\|_{L^\infty(\M)}(t)\leq K_2(\T),
$$
for a continuous bounded function $K_2(\T)$ determined by the quantities
$\|\overline\omega\|_{L^\infty(0,\T;L^2(\M))},$ $\|T_0\|_{L^\infty{(\M)}},$ $\|T_{b0}\|_{L^\infty((0,T)\times\M')}$, $\|T_{b\ll}\|_{L^\infty((0,\T)\times\Gamma_\ll)}$, $q_v^*$, $q_c^*,$ and $q_r^*$.
\end{prop}

\begin{proof}
In \cite{HKLT},
the upper bound for the temperature was derived by employing the potential temperature equation. Again here this does not alleviate the computations due to the stronger coupling of the thermodynamic equation (\ref{eq.T}) to the moisture quantities.
We thus instead apply here the proof of Coti Zelati et al.\,\cite{CZH} based on the De Giorgi technique.

Let $\lambda_k\geq \max\{\|T_0\|_{L^\infty(\M)},T_{\text{crit}},\|T_{b\ll}\|_{L^\infty((0,\T)\times\Gamma_\ll)},
\|T_{b\ll}\|_{L^\infty((0,\T)\times\Gamma_\ll)}\}$, and denote
$T_{\lambda_k} =(T-\lambda_k)^+.$
We claim that
\begin{equation}
  \widetilde LS_{cd}T_{\lambda_k}\leq0.
  \label{EQ-2}
\end{equation}
In fact, if $T<\lambda_k$, then $T_{\lambda_k}=(T-\lambda_k)^+=0$ and, as a result, $\widetilde LS_{cd}T_{\lambda_k}=0$, while if
$T\geq\lambda_k$, it is clear by the definition of $\lambda_k$ that $T\geq T_{\text{crit}}$, and as a result, noticing that in this case $L(T)\leq0$ (due to (\ref{L}) and (\ref{Tcrit})) and $q_{vs}(p,T)=0$ (due to (\ref{cond.qvs.up})),
it holds that
\begin{equation}
\label{EQ-2'}
  \widetilde LS_{cd}T_{\lambda_k}=\frac{L(T)}{\widetilde C}[C_{cd}(q_v-q_{vs})q_c+C_{cn}(q_v-q_{vs})^+]T_{\lambda_k}
  \leq0.
\end{equation}
Testing equation (\ref{eq.T}) with $T_{\lambda_k}$ and by similar calculations as (\ref{est.diff.q}) and (\ref{AD4})
to the diffusion terms and the convection terms, we have the following estimate
\begin{eqnarray}
&&\frac{1}{2}\frac{d}{dt}\|T_{\lambda_k}\|^2 + \mu_T\|\nabla_h T_{\lambda_k}\|^2 +\nu_T\|\pa_p T_{\lambda_k}\|_w^2\nonumber \\
&\leq& \int_\M \frac{\widetilde{\kappa}}{p} \overline{\omega}  TT_{\lambda_k} d\M
- \int_\M \frac{c_l q_r}{\widetilde{C}}V \pa_p T_{\lambda_k} \, T_{\lambda_k} d\M + \int_\M \widetilde{L}(S_{cd}-S_{ev})  T_{\lambda_k}d\M\,\nonumber\\
&\leq& C\int_\M |\overline{\omega}| (T_{\lambda_k}^2 + {\lambda_k} T_{\lambda_k}) d\M+C\int_\M | \pa_p T_{\lambda_k}|T_{\lambda_k} d\M,
\label{61.0}
\end{eqnarray}
where (\ref{bound.kappa}), (\ref{EQ-1}), (\ref{EQ-2}), and (\ref{EQ-2'}) were used in the last step.
By Lemma \ref{aux.lem.uni} and the Young inequality, it follows that
$$
  C\int_\M(|\bar\omega|T_{\lambda_k}^2+|\partial_pT_{\lambda_k}|T_{\lambda_k})d\M
  \leq\frac12\left(\mu_T\|\nabla_hT_{\lambda_k}\|^2+\nu_T\|\partial_pT_{\lambda_k}\|_w^2\right)+C\|T_{\lambda_k}\|^2,
$$
where $C$ depends on $\|\bar\omega\|_{L^\infty(0,\T; L^2(\M))}$,
which substituted into (\ref{61.0}) yields
$$\frac{d}{dt}\|T_{\lambda_k}\|^2 + \mu_T\|\nabla_h T_{\lambda_k}\|^2 +\nu_T\|\pa_p T_{\lambda_k}\|_w^2
\leq C\|T_{\lambda_k}\|^2+C{\lambda_k}\int_\M|\bar\omega|  T_{\lambda_k} d\M.
$$
Due to $T_{\lambda_k}|_{t=0}=0$, we obtain by applying the Gr\"onwall inequality to the above
\begin{eqnarray}
 J_k&:=&\sup_{t\in(0, \T)} \|T_{\lambda_k}\|^2(t)+ \int_0^\T(\mu_T\|\nabla_h T_{\lambda_k}\|^2 +\nu_T\|\pa_p T_{\lambda_k}\|_w^2) dt\nonumber\\
&\leq& C{\lambda_k}\int_0^\T\int_\M|\bar\omega|  T_{\lambda_k} d\M dt.\label{61.2}
\end{eqnarray}

Let $M\geq 2\max\{\|T_0\|_{L^\infty(\M)},T_{\text{crit}},\|T_{b\ll}\|_{L^\infty((0,\T)\times\Gamma_\ll)},
\|T_{b\ll}\|_{L^\infty((0,\T)\times\Gamma_\ll)}\}$ be a positive constant to be determined later and choose
$$
 \lambda_k := M(1-2^{-k}), \quad Q_k:=\left\{(x,t)\in\M\times(0,\T)|T(x,t)>\lambda_k\right\},\quad \textnormal{for} \quad k\geq 1.
$$
For any $(x,t)\in Q_k$, noticing that $T(x,t)>\lambda_k>\lambda_{k-1}$, one deduces
\begin{eqnarray*}
  T_{\lambda_{k-1}}(x,t)=(T-\lambda_{k-1})^+(x,t)=T(x,t)-\lambda_{k-1}>\lambda_k-\lambda_{k-1}=2^{-k}M
\end{eqnarray*}
and, thus,
$$
  \chi_{Q_k}(x,t)\leq\frac{2^k}{M}T_{\lambda_{k-1}}(x,t).\label{61.4}
$$
Thanks to this and noticing that $T_{\lambda_k}\leq T_{\lambda_{k-1}}$ and $\lambda_k\leq M$, one deduces from (\ref{61.2}) and the Gagliardo-Nirenberg inequality that
\begin{eqnarray*}
  J_k&\leq&CM\int_0^{\T}\int_\M|\bar\omega|T_{\lambda_k}d\M dt=CM\int_0^\T\int_\M|\bar\omega|T_{\lambda_k}\chi_{Q_k}^\frac43d\M dt\\
  &\leq&\frac C{M^\frac13}16^{\frac k3}\int_0^\T\int_\M|\bar\omega|T_{\lambda_{k-1}}^\frac73d\M dt\leq \frac C{M^\frac13}16^{\frac k3}\int_0^\T \|\bar\omega\|\|\T_{\lambda_{k-1}}\|_{L^\frac{14}{3}(\M)}^\frac73  dt\\
  &\leq& \frac C{M^\frac13}16^{\frac k3}\int_0^\T \|\T_{\lambda_{k-1}}\|^\frac13\left( \|\T_{\lambda_{k-1}}\|^2+\|\nabla T_{\lambda_{k-1}}\|^2\right) dt\\
  &\leq&\frac{C_*}{M^\frac13}16^\frac k3J_{k-1}^\frac76, \qquad k=2,3,4,\cdots,
\end{eqnarray*}
that is
\begin{equation}
  J_k\leq \frac{C_*}{M^\frac13}16^\frac k3J_{k-1}^\frac76, \qquad k=2,3,4,\cdots,\label{61.5}
\end{equation}
where $C$ depends on $\|\bar\omega\|_{L^\infty(0,\T; L^2(\M))}$.
Setting
$$
  a=64^7C_*^6M^{-2},\quad b=64,\quad S_k=ab^kJ_k,
$$
and by simple calculations, one can check from (\ref{61.5}) that
\begin{equation*}
  S_{k+1}=ab^{k+1}J_{k+1}\leq(ab^kJ_k)^\frac76=S_k^{\frac76}
\end{equation*}
and, thus,
\begin{equation}
  \label{61.6}
64^{k+8}C_*^6M^{-2}J_{k+1}=S_{k+1}\leq S_1^{\left(\frac76\right)^k}=\left[64^8C_*^6M^{-2}J_1\right]^{\left(\frac76\right)^k},\quad k=1,2,\cdots.
\end{equation}
Recalling the definition of $J_k$ and applying Proposition \ref{L2T} lead to
\begin{eqnarray}
  J_1&=&\sup_{t\in(0, \T)} \|T_{\lambda_1}\|^2(t)+ \int_0^\T(\mu_T\|\nabla_h T_{\lambda_1}\|^2 +\nu_T\|\pa_p T_{\lambda_1}\|_w^2) dt\nonumber\\
  &\leq&\sup_{t\in(0, \T)} \|T\|^2(t)+ \int_0^\T(\mu_T\|\nabla_h T\|^2 +\nu_T\|\pa_p T\|_w^2) dt\leq C_{**}. \label{61.6-1}
\end{eqnarray}
Choose $M$ large enough such that $M\geq 2\max\{\|T_0\|_{L^\infty(\M)},T_{\text{crit}},\|T_{b\ll}\|_{L^\infty((0,\T)\times\Gamma_\ll)},
\|T_{b\ll}\|_{L^\infty((0,\T)\times\Gamma_\ll)}\}$ and $64^8C_*^6M^{-2}C_{**}\leq\frac12$. Then, it follows from (\ref{61.6}) and (\ref{61.6-1}) that
\begin{equation*}
  64^{k+8}C_*^6M^{-2}J_{k+1}=S_{k+1}\leq\left(\frac12\right)^{\left(\frac76\right)^k}
\end{equation*}
and, thus,
$\lim_{k\rightarrow\infty}J_k=0.$ This leads to the desired bound $T\leq M$ on $\M\times(0,\T)$.
\end{proof}

\begin{prop}
\label{prop.H1}
If Assumption \ref{ass} holds, we have the estimates
  \begin{eqnarray*}
  \sup_{0\leq t\leq\mathcal T}\|\nabla(q_v,q_c,q_r,T)\|^2(t)
+\int_0^\mathcal T\|\nabla^2(q_v,q_c,q_r,T)\|^2(t)dt\leq K_3(\T),
\end{eqnarray*}
for a continuous bounded function $K_3(\T)$ determined by the quantities
 $q_v^*, q_c^*, q_r^*$, $K_2(\T)$, $\|T_0\|_{L^\infty{(\M)}}$, $\|(T_0, q_{v0}, q_{r0}, q_{c0})\|_{H^1(\M)}$, and $\|(\overline{\mathbf{v}}_h,\overline\omega)\|_{L^r(0,\T; L^q(\M))}.$
\end{prop}

\begin{proof}
We only give the details about the proof for the estimate of $T$, those for the moisture components are similar (actually
simpler).

We first estimate the vertical derivative $\partial_pT$. Multiplying the thermodynamic equation by $-\partial_p^2T$ and integrating the resultant over $\mathcal M$ yields
\begin{eqnarray}
&&\int_\mathcal M(-\partial_tT+\mathcal D^TT)\partial_p^2Td\mathcal M
\nonumber\\
&=&\int_\mathcal M\left[\overline{\mathbf{v}}_h\cdot\nabla_hT+\overline\omega\partial_pT
  -\widetilde \kappa\frac{T}{p}\overline\omega+\frac{c_lq_r}{\widetilde C} V \pa_p T+\widetilde L(S_{cd}-S_{ev})\right]
  \partial_p^2Td\mathcal M.\label{est.pt.0}
\end{eqnarray}
Following the derivations in \cite{HKLT2} (see (87) and (88) there) we have
\begin{equation}
-\int_\mathcal M\partial_tT\partial_p^2Td\mathcal M
\geq\frac{d}{dt}\left(\frac{\|\partial_pT\|^2}{2}+\alpha_{0T}
\left.\int_{\mathcal M'}\left(\frac{T^2}{2}-TT_{b0}\right) d\mathcal M'\right|_{p_0}\right)-C(\|\partial_pT\|+C_3), \label{est.pt.1}
\end{equation}
\begin{equation}
\int_\mathcal M\mathcal D^{T}T\partial_p^2Td\mathcal M
\geq \frac34\left(\mu_{T}\|\nabla_h\partial_pT\|^2+\nu_{T}
\|\partial_p^2T\|_w^2\right)-C(\|\partial_pT\|^2+C_4). \label{est.pt.2}
\end{equation}
By the H\"older, Sobolev, and Young inequalities, one deduces
\begin{eqnarray}
  &&\int_\mathcal M(\overline{\mathbf{v}}_h\cdot\nabla_hT+\overline\omega\partial_pT)
\partial_p^2Td\mathcal M\nonumber\\
&\leq&C\|(\overline{\mathbf{v}}_h,\overline\omega)\|_{L^q(\M)}\|\nabla T\|_{L^{\frac{2q}{q-2}}(\M)}\|\partial_p^2T\|\leq C\|(\overline{\mathbf{v}}_h,\overline\omega)\|_{L^q(\M)}\|\nabla T\|^{1-\frac3q}\|\nabla T\|_{H^1(\M)}^{1+\frac3q}\nonumber\\
&\leq&
\eta\|\nabla^2T\|^2+C_\eta\left(\|(\overline{\mathbf{v}}_h,\overline\omega)\|_{L^q(\M)}^{\frac{2q}{q-3}}+C_5\right)\|\nabla T\|^2,\label{est.pt.3}
\end{eqnarray}
for any positive number $\eta$.
Moreover, by (\ref{bound.kappa}) and Proposition \ref{LinftyT}, we obtain by the Young inequality that
\begin{eqnarray}
-\int_\mathcal M\widetilde\kappa\frac{T}{p}\overline\omega\partial_p^2Td\mathcal M+\int_\mathcal M \frac{c_l q_r}{\widetilde C} V \pa_pT \pa_p^2 T d\mathcal M
 \leq\eta\|\partial_p^2T\|_w^2+C_\eta(\|\overline\omega\|^2+C_6\|\partial_pT\|^2),\label{est.pt.4}
\end{eqnarray}
for any positive number $\eta$.
Due to the nonnegativity and uniform boundedness of $T, q_v, q_c, q_r$, it follows from the Young inequality that
\begin{equation}
  \int_\mathcal M\widetilde{L}(S_{cd}-S_{ev})\partial_p^2Td\mathcal M
  \leq C\int_\mathcal M|\partial_p^2T|d\mathcal M
  \leq\eta\|\partial_p^2T\|^2+C_\eta,\label{est.pt.5}
\end{equation}
for any positive number $\eta$.
Substituting (\ref{est.pt.1})--(\ref{est.pt.5}) into (\ref{est.pt.0}), one obtains
\begin{eqnarray}
\frac{d}{dt}\left(\frac{\|\partial_pT\|^2}{2}+\alpha_{0T}
\left.\int_{\mathcal M'}\left(\frac{T^2}{2}-TT_{b0}\right) d\mathcal M'\right|_{p_0}\right)+  \frac34(\mu_{T} \|\nabla_h\partial_pT\|^2+\nu_{T} \|\partial_p^2T\|_w^2) \nonumber\\
\leq 3\eta\|\nabla^2T\|^2+C_\eta\left(\|(\overline{\mathbf{v}}_h,\overline\omega)\|_{L^q(\M)}^{\frac{2q}{q-3}}
+\|\overline\omega\|^2+C_7\right)(\|\nabla T\|^2+C_8),\label{EST.PT}
\end{eqnarray}
for any positive number $\eta$.

Next, we estimate the horizontal gradient $\nabla_hT$. Multiplying equation (\ref{AT}) by $-\Delta_hT$ and integrating the resultant over $\mathcal M$ yields
\begin{eqnarray}
  &&\int_\mathcal M(-\partial_tT+\mathcal D^TT)\Delta_hTd\mathcal M
\nonumber\\
&=&\int_\mathcal M\left[\overline{\mathbf{v}}_h\cdot\nabla_hT+\overline\omega\partial_pT
  -\widetilde\kappa\frac{T}{p}\overline\omega+  \frac{c_l q_r}{\widetilde C} V \pa_pT+\widetilde L(S_{cd}-S_{ev})\right]
  \Delta_hTd\mathcal M.\label{est.ht.0}
\end{eqnarray}
Following the derivations in \cite{HKLT2} (see (94) and (95) there) we have
\begin{eqnarray}
-\int_\mathcal M\partial_tT\Delta_hTd\mathcal M&\geq&\frac{d}{dt}\left(\frac{\|\nabla_hT\|^2}{2}+\alpha_{\ell T}
\int_{\Gamma_\ell}\left(\frac{T^2}{2}-TT_{b\ell}\right)d\Gamma_\ell
\right)\nonumber\\
&&-C(C_9+\|\nabla_hT\|),\label{est.ht.1}\\
\int_\mathcal M\mathcal D^TT\Delta_hTd\mathcal M
&\geq&\mu_T\|\Delta_hT\|^2+\nu_T\|\nabla_h\partial_pT\|_w^2
-C.\label{est.ht.2}
\end{eqnarray}
Similar to (\ref{est.pt.3})--(\ref{est.pt.5}), we have
\begin{equation}
\int_\mathcal M(\overline{\mathbf{v}}_h\cdot\nabla_hT+\overline\omega\partial_pT)
\Delta_hTd\mathcal M\leq
\eta\|\nabla^2T\|^2+C_\eta\left(
\|(\overline{\mathbf{v}}_h,\overline\omega)\|_{L^q(\M)}^{\frac{2q}{q-3}}+C_{10}\right)\|\nabla T\|^2,
\label{est.ht.3}
\end{equation}
\begin{equation}
-\int_\mathcal M\widetilde\kappa\frac{T}{p}\overline\omega\Delta_hTd\mathcal M+
\int_\mathcal M \frac{c_l q_r}{\widetilde C} V \pa_pT\Delta_h T d\M\leq\eta\|\Delta_hT\|^2+C_\eta(\|\overline\omega\|^2+C_{11}\|\partial_pT\|_w^2),\label{est.ht.4}
\end{equation}
\begin{equation}
  \int_\mathcal M \widetilde L(S_{cd}-S_{ev})\Delta_hTd\mathcal M
  \leq\eta\|\Delta_hT\|^2+C_\eta,\label{est.ht.5}
\end{equation}
for any positive number $\eta$.
Substituting (\ref{est.ht.1})--(\ref{est.ht.5}) into (\ref{est.ht.0}) gives
\begin{eqnarray}
  \frac{d}{dt}\left(\frac{\|\nabla_hT\|^2}{2}+\alpha_{\ell T}
\int_{\Gamma_\ell}\left(\frac{T^2}{2}-TT_{b\ell}\right)d\Gamma_\ell
\right)+\frac34(\mu_T\|\Delta_hT\|^2+\nu_T\|\nabla_h\partial_pT\|_w^2)\nonumber\\
\leq 3\eta\|\nabla^2T\|^2+C_\eta\left(\|(\overline{\mathbf{v}}_h,\overline\omega)\|_{L^q(\M)}^{\frac{2q}{q-3}}
+\|\overline\omega\|^2+C_{12}\right)(\|\nabla T\|^2+C_{13}),\label{EST.HT}
\end{eqnarray}
for any positive number $\eta$.

Summing (\ref{EST.PT}) with (\ref{EST.HT}) yields
\begin{eqnarray}
&&\frac34\Big(\mu_{T} \|\nabla_h\partial_pT\|^2+\nu_{T} \|\partial_p^2T\|_w^2+\mu_T\|\Delta_hT\|^2+\nu_T\|\nabla_h\partial_pT\|_w^2\Big)
\nonumber\\
&&+\frac{d}{dt}\left(\frac{\|\nabla T\|^2}{2}+\alpha_{0T}
\left.\int_{\mathcal M'}\left(\frac{T^2}{2}-TT_{b0}\right) d\mathcal M'\right|_{p_0}
+\alpha_{\ell T}
\int_{\Gamma_\ell}\left(\frac{T^2}{2}-TT_{b\ell}\right)d\Gamma_\ell
\right)\nonumber\\
&\leq& 6\eta\|\nabla^2T\|^2+C_\eta\left(\|(\overline{\mathbf{v}}_h,\overline\omega)\|_{L^q(\M)}^{\frac{2q}{q-3}}
+\|\overline\omega\|^2+C_{14}\right)(\|\nabla T\|^2+C_{15}),\label{EST.NT}
\end{eqnarray}
for any positive number $\eta$. Applying the elliptic estimate (see Proposition A.2 in \cite{HKLT}) to the elliptic equation $\cal D^TT=f$ subject to the boundary condition (\ref{bound.0})--(\ref{bound.ll}) leads to
$$
\|\nabla^2T\|\leq C(\|\cal D^TT\|+C_{16}\|\nabla T\|+C_{17})\leq C(\|\Delta_hT\|+\|\partial_p^2T\|_w^2+C_{18}\|\nabla T\|+C_{19}).
$$
Thanks to this, and noticing that
\begin{eqnarray*}
\int_{\mathcal M'}\left(\frac{T^2}{2}-TT_{b0}\right)d\mathcal M'
=\frac12 \int_{\mathcal M'}\left[(T-T_{b0})^2-T_{b0}^2\right]d\mathcal M'
\geq- \int_{\mathcal M'}T_{b0}^2 d\mathcal M' \geq-C,
\end{eqnarray*}
and
\begin{eqnarray*}
\int_{\Gamma_\ell}\left(\frac{T^2}{2}-TT_{b\ell}
\right)d
\Gamma_\ell&=&\frac12\int_{\Gamma_\ell}\left[(T-T_{b\ell})^2-T_{b\ell}^2
\right]d
\Gamma_\ell\geq
-\frac12\int_{\Gamma_\ell}T_{b\ell}^2
d\Gamma_\ell\geq-C,
\end{eqnarray*}
the conclusion follows by applying the Gr\"onwall inequality to (\ref{EST.NT}).
\end{proof}

We are now ready to give the proof of our main result, Theorem \ref{thm}.

\begin{proof}[Proof of Theorem \ref{thm}]
By Proposition \ref{Aloc}, there is a unique
local solution $( T, q_v, q_c, q_r)$ to system (\ref{AT})--(\ref{Aqr}), subject to
(\ref{bound.0})--(\ref{bound.ll}). Due to Proposition \ref{bddq} and Proposition \ref{nonneg.T}, $q_c, q_c, q_r,$ and $T$
are all nonnegative and, thus,
$(T, q_v, q_c, q_r)$ is a local solution to the original system, subject to the corresponding
initial and boundary conditions. By applying Proposition \ref{Aloc} inductively, one can extend the local solution
to the maximal time of existence $\T_{\text{max}}$ characterized by (\ref{T_max}).
We need to prove $\mathcal T_{\text{max}}=\infty$. Assume, by contradiction, that $\mathcal T_{\text{max}}<\infty$.
Then, by Propositions \ref{bddq}--\ref{prop.H1}, we have the estimate
\begin{align*}
  \sup_{0\leq t\leq\mathcal T}
  \|(T,q_v, q_c, q_r)\|_{H^1(\mathcal M)}\leq C_0,
\end{align*}
for any $\mathcal T\in(0,\mathcal T_{\text{max}})$, and $C_0$ is a positive constant independent of
$\mathcal T\in(0,\mathcal T_{\text{max}})$. This contradicts (\ref{T_max}) and, thus,
$\T_{\text{max}}=\infty$, proving the conclusion.
\end{proof}

\noindent\textbf{Acknowledgments.} S.H. acknowledges support by the Austrian Science Fund via
the previous Hertha-Firnberg project T-764 and via the SFB ``Taming Complexity in Partial Differential Systems'' with project number F 65. R.K.\ acknowledges support by Deutsche Forschungsgemeinschaft through
Grant CRC 1114 ``Scaling Cascades in Complex Systems'', projects A02 and C06. J.L. acknowledges support by the National Natural Science Foundation of
China (11971009 and 11871005), by the Guangdong Basic and Applied Basic Research Foundation (2019A1515011621, 2020B1515310005, 2020B1515310002, and 2021A1515010247), and by the Key Project of National Natural Science Foundation of China (12131010). E.S.T.acknowledges support by the Einstein Stiftung/Foundation-Berlin, Einstein Visiting Fellowship No. EVF-2017-358.
RK thanks the Centre International de Rencontres Math\'ematique (CIRM) and
the city of Marseille for their support in the framework of their Jean
Morlet Chair programme on ``Nonlinear partial differential equations in
fluid mechanics". RK and EST would also like to thank the Isaac Newton
Institute for Mathematical Sciences for support and hospitality during
the programme TUR when part of this work was undertaken. This work was
supported by EPSRC Grant Number EP/R014604/1.




\begin{thebibliography}{99}

\bibitem{B} P.~R.~Bannon. \emph{Theoretical Foundations for Models of Moist Convection.} JAS 59:1967--1982, (2002)


\bibitem{BCZ} A.~Bousquet, M.~Coti~Zelati,  R.~Temam. \emph{Phase transition models in atmospheric dynamics.} Milan Journal of Mathematics
82: 99--128, (2014)

\bibitem{CLT3} C.~Cao, J.~Li, E.S.~Titi. \emph{Global well-posedness of the 3D primitive equations with horizontal viscosity and vertical diffusivity}, Phys. D, 412 (2020), 132606, 25 pp.

\bibitem{CLT} C.~Cao, J.~Li, E.S.~Titi. \emph{Global well posedness for the 3D primitive equations with only horizontal viscosity and diffusion,} Comm. Pure Appl. Math., 69: 1492--1531, (2016).

\bibitem{CLT1} C.~Cao, J.~Li, E.S.~Titi. \emph{Global well-posedness of strong solutions to the 3D primitive equations with horizontal eddy diffusivity,} Journal of Differential Equations, 257: 4108--4132, (2014)

\bibitem{CLT2} C.~Cao, J.~Li, E.S.~Titi. \emph{Local and global well-posedness of strong solutions to the 3D primitive equations with vertical eddy diffusivity,} Arch. Ration. Mech. Anal., 214: 35--76, (2014).

\bibitem{CT} C.~Cao, E.S.~Titi. \emph{Global well-posedness of the three-dimensional viscous primitive equations of large scale ocean and atmosphere dynamics.} Annals of Mathematics, 166(1):245--267, (2007)



\bibitem{CZF} M.~Coti~Zelati, M.~Fr\'emond, R.~Temam, J.~Tribbia. \emph{The equations of the atmosphere with humidity and saturation: uniqueness and physical bounds.} Physica D, 264: 49--65, (2013)

\bibitem{CZH} M.~Coti~Zelati, A.~Huang, I.~Kukavica, R.~Temam,  M.~Ziane. \emph{The primitive equations of
the atmosphere in presence of vapor saturation,} Nonlinearity, 28 (3): 625--668, (2015).

\bibitem{CZT} M.~Coti Zelati, R.~Temam. \emph{The atmospheric equation of water vapor with saturation.}
Bollettino dell'Unione Matematica Italiana, 5: 309--336, (2012)

\bibitem{C} W.R. Cotton, G. Bryan, S.C. van den Heever. \emph{Storm and Cloud Dynamics.} Second edition (International Geophysics), Academic Press, (2011)

%
%

\bibitem{Fri} A.~Friedman. Partial differential equations. Holt, Rinehart and Winston, Inc., New York-Montreal, Que.-London, 1969.


\bibitem{GS} W.W.~Grabowski., P.~K.~Smolarkiewicz \emph{Two-Time-Level Semi-Lagrangian Modeling of Precipitating Clouds.} Monthly Wea. Reviews, 124: 487-497, (1996)


\bibitem{HK} S.~Hittmeir, R.~Klein. \emph{Asymptotics for moist deep convection I: Refined scalings and self-sustaining updrafts}, Theoretical and Computational Fluid Dynamics, Vol. 32 (2): 137--164, (2017)

\bibitem{HKLT} S.~Hittmeir, R.~Klein, J.~Li, E.~Titi. \emph{Global well-posedness for passively transported nonlinear moisture dynamics with phase changes}, Nonlinearity 30: 3676--3718, (2017)

\bibitem{HKLT2} S.~Hittmeir, R.~Klein, J.~Li, E.~Titi. \emph{Global well-posedness for the primitive equations coupled to nonlinear moisture dynamics with phase changes}, Nonlinearity, 33 (7): 3206--3236, (2020).

\bibitem{HW} G.~J.~Haltiner, R.~T.~Williams. \emph{Numerical prediction and dynamic meteorology.} John Wiley and Sons, 2nd edition, (1980)



\bibitem{Ke} E.~Kessler. \emph{On the distribution and continuity of water substance in atmospheric circulations.} Meteorol. Monogr. 10(32), (1969)

    \bibitem{BoualemsBook2019} B. Khouider, \emph{Models for Tropical Climate Dynamics-Waves, Clouds, and Precipitation}, Mathematics of Planet Earth, 3, Springer (2019)


\bibitem{KM} R.~Klein, A.~J.~Majda. \emph{Systematic Multiscale Models for Deep Convection on Mesoscales.} Theoretical and Computational Fluid Dynamics 20 (5-6): 525--551, (2006)






\bibitem{LTW} J.L.~Lions, R.~Temam, S.~Wang. \emph{New formulations of the primitive equations of atmosphere and applications.} Nonlinearity 5: 237--288, (1992)


\bibitem{MXM} A.~J.~Majda, Y.~Xing, M.~Mohammadian. \emph{Moist multi-scale models for the
hurricane embryo}, J. Fluid Mech. (2010), vol. 657, pp. 478-501.

\bibitem{PTZ} M.~Petcu, R.~M.~Temam, M.~Ziane. \emph{Some mathematical problems in geophysical fluid dynamics.} Handbook of Numerical Analysis Vol. XIV

%

\bibitem{S} R.~K.~Smith (ed.) \emph{The physics and parameterization of moist atmospheric convection}, Kluwer Academic Publisher, (1997)

    \bibitem{SmithStechmann2017} L.~Smith, S.~N.~Stechmann. \emph{Precipitating quasigeostrophic equations and potential vorticity inversion with phase changes}, J. Atmos. Sci. 74: 3285--3303, (2017)


\bibitem{StechmannHottovy2020} S.~N.~Stechmann, S.~Hottovy, \emph{Asymptotic models for tropical intraseasonal oscillations and geostrophic balance}, J. Climate 33: 4715--4737, (2020)

\bibitem{Zhe} S.~Zheng. {\em Nonlinear Parabolic Equations and Hyperbolic-Parabolic Coupled Systems}. Pitman, New York, 1995.


\end{thebibliography}
\end{document}